\documentclass[12pt]{amsart}
\usepackage{amssymb}

\theoremstyle{plain}
\newtheorem{thm}[subsection]{Theorem}
\newtheorem{cor}[subsection]{Corollary}
\newtheorem{prop}[subsection]{Proposition}
\newtheorem{lemma}[subsection]{Lemma}
\newtheorem*{ABC}{ABC Theorem}
\theoremstyle{remark}
\newtheorem{rem}[subsection]{Remark}

\newtheorem{exs}[subsection]{Examples}

\usepackage[OT2,T1]{fontenc}
\DeclareSymbolFont{cyrletters}{OT2}{wncyr}{m}{n}
\DeclareMathSymbol{\sha}{\mathalpha}{cyrletters}{"58}

\newcommand{\CC}{\mathcal{C}}

\newcommand{\XX}{\mathcal{X}}

\newcommand{\YY}{\mathcal{Y}}

\newcommand{\F}{\mathbb{F}}

\newcommand{\Fq}{{\mathbb{F}_q}}

\newcommand{\Fqbar}{{\overline{\mathbb{F}}_q}}

\newcommand{\ratto}{{\dashrightarrow}}

\renewcommand{\P}{\mathbb{P}}

\def\clap#1{\hbox to 0pt{\hss#1\hss}}

\date{December 14, 2016}

\begin{document}

\title[Number of points]{On the number of rational points \\
on  special families of curves \\ over function fields}

\author{Douglas Ulmer}
\address{School of Mathematics, 
Georgia Institute of Technology, 
Atlanta, GA 30332, USA}
\email{douglas.ulmer@math.gatech.edu}

\author{Jos\'e Felipe Voloch}
\address{School of Mathematics and Statistics,
University of Canterbury, 
Private Bag 4800, Christchurch 8140, New Zealand}
\email{felipe.voloch@canterbury.ac.nz}


\begin{abstract}
  We construct families of curves which provide counterexamples for a
  uniform boundedness question.  These families generalize those studied
  previously by several authors in \cite{Ulmer14a}, \cite{AIMgroup},
  and \cite{ConceicaoUlmerVoloch12}.  We show, in detail, what fails
  in the argument of Caporaso, Harris, Mazur that uniform boundedness
  follows from the Lang conjecture. We also give a direct proof that
  these curves have finitely many rational points and give explicit
  bounds for the heights and number of such points.
\end{abstract}

\maketitle

\section{Unboundedness of rational points}

The question of whether there is a uniform bound for the number of
rational points on curves of fixed genus greater than one over a fixed
number field has been considered by several authors.  In particular,
in \cite{CaporasoHarrisMazur97} Caporaso et al.\ showed that this
would follow from the Bombieri-Lang conjecture that the set of
rational points on a variety of general type over a number field is
not Zariski dense.  In \cite{ConceicaoUlmerVoloch12}, Concei\c c\~ao
and the present authors gave examples over function fields of families
of smooth curves of fixed genus whose number of rational points is
unbounded.  Our first point is that these examples are part of a more
general family.

Fix a prime $p$ and a power $q$ of $p$, let $\Fq$ be the field of $q$
elements, and let $\Fq(t)$ be the rational function field over $\Fq$.
Choose an integer $r>1$ and prime to $p$, and let $h(x)\in\Fq[x]$ be
a polynomial of positive degree which is not the $e$-th power of another
element of $\Fq(t)$ for any divisor $e>1$ of $r$.  We also assume that
$h(0)\neq0$.  For $a\in\Fq(t)\setminus\Fq$, let $X=X_{h,r,a}$ be the
smooth projective curve over $\Fq(t)$ associated to the equation
$$X:\quad y^r=h(x)h(a/x).$$
Our hypotheses imply that $X$ is absolutely irreducible and its genus
is independent of $a$.  In the case where $h$ has distinct roots and
degree $s$ with $(r,s)=1$, the Riemann-Hurwitz formula shows that $X$
has genus $g=(r-1)s$.

\begin{thm}
  Assume that $r$ divides $q^f+1$ for some $f\ge1$.  Then as $a$
  varies through $\Fq(t)\setminus\Fq$, the number of rational points
  of the curve $X_{h,r,a}$ over $\F_q(t)$ is unbounded.
\end{thm}

\begin{proof} 
  We first note that if $d=q^n+1$, $r$ divides $d$, and $a=t^d$, then
  we have a rational point $(x,y)=(t,h(t)^{d/r})$ on $X$.  Second, if
  $m$ divides $n$ and $n/m$ is odd, then $d'=q^m+1$ divides $d$.  If
  $r$ divides $d'$, setting $e=d/d'$, we have another rational point
  $(x,y)=(t^e,h(t^{e})^{d'/r})$ on $X$.  Thus if we take $n$ to be a
  multiple of $f$ such that $n/f$ is odd and has many factors, we have
  many points.
\end{proof}

\begin{exs} Up to change of coordinates, the case $r=2$, $h(x)=x+1$ is
  the elliptic curve studied in \cite{Ulmer14a}, the case $r>1$,
  $h(x)=x+1$ is the curve of genus $r-1$ whose Jacobian is the subject
  of \cite{AIMgroup}, and the case $r=2$, $h(x)=x^s+1$ with $s$ odd is
  the curve of genus $s$ studied in \cite{ConceicaoUlmerVoloch12}.
\end{exs}

\begin{rem} 
  Fixing $h$ and $r$, here we consider a family of curves $X_a$ over a
  fixed field $\Fq(t)$.  It is sometimes more convenient to consider
  the fixed curve $y^r=h(x)h(t/x)$ over extensions
  $\Fq(u)/\Fq(t)$ where $t$ is a varying rational function of $u$.
\end{rem}

Consider the case $r=2$, $h(x)=x^s+1$ with $s$ odd.  Let $\XX$ be the
smooth projective surface with affine model
$$y^2=x(x^s+1)(x^s+t^s)$$  
and consider the fibration $\XX \to \P^1$, $(x,y,t) \mapsto t$.  Its
generic fiber is isomorphic to $X_{h,r,t}$ over $\Fq(t)$.  As remarked
in \cite{ConceicaoUlmerVoloch12}, the results of
\cite{CaporasoHarrisMazur97} show that the fibration has a fibered
power which covers a variety of general type.  However, since this
fibration is defined over a finite field, the variety of general type
will also be defined over a finite field.   Moreover it may have a
Zariski dense set of $\F_p(t)$-rational points, so the rest of the
argument of \cite{CaporasoHarrisMazur97} does not apply.  (See
\cite{AbramovichVoloch96} for a general discussion, including the
function field case).

We can be more specific: In the next section, we will see that for
many choices of $h$ and $r$, $X_{h,r,t}$ has a model over $\P^1_t$
which is already a variety of general type.

\section{Geometry of a regular proper model of $X$}
When a curve $X$ over $\Fq(t)$ has a model $\XX\to\P^1_t$ such that
$\XX$ is dominated by a product of curves, many questions about $X$
become much simpler.  For example, the Tate conjecture on divisors
holds for $\XX$, the conjecture of Birch and Swinnerton-Dyer holds for
the Jacobian of $X$, and it is often possible to compute or estimate
the rank of the N\'eron-Severi group of $\XX$ and the rank of group of
rational points on the Jacobian.  (This observation is mainly due to
Shioda \cite{Shioda86} with further elaboration in \cite{Ulmer07b}.)

Fix a polynomial $h(x)\in\Fq[t]$ and an integer $r$ with
hypotheses as in the first section.  Fix also an integer $d$ prime to
$p$, and let $X=X_{h,r,t^d}$ be the smooth projective curve over
$\Fq(t)$ associated to the equation
$$y^r=h(x)h(t^d/x).$$
Let $\XX$ be a smooth projective surface equipped with a morphism to
$\P^1$ whose generic fiber is isomorphic to $X$.  (The construction is
elementary; see \cite[Ch.~2]{Ulmer14b} for details.)  In this section,
we will show that $\XX$ is dominated by a product of curves and give
two applications:  $\XX$ is often of general type, and $X$ is
non-isotrivial.  

Let $\CC=\CC_{h,r,d}$ be the smooth projective curve over $\Fq$
associated to 
$$w^r=h(z^d).$$
Our hypotheses on $h$ and $r$ imply that $\CC$ is absolutely
irreducible.  Note that $\CC$ admits an action (over $\Fqbar$) of the
group $G:=\mu_r\times\mu_d$.

\begin{prop}
The surface $\XX$ is birational to the quotient of $\CC\times\CC$ by
the action of $G$, where $G$ acts ``anti-diagonally,'' i.e., by the
action above on the first factor and by its inverse on the second
factor.
\end{prop}

\begin{proof}
The surface $\XX$ is birational to the (quasi-) affine surface given
by 
$$\YY:\quad y^r=h(x)h(t^d/x).$$
We define a rational map $\phi$ from $\CC\times\CC$ to $\YY$ by
setting
\begin{align*}
\phi^*(x)&=z_1^d\\
\phi^*(y)&=w_1w_2\\
\phi^*(t)&=z_1z_2
\end{align*}
It is evident that $\phi$ factors through $(\CC\times\CC)/G$ where $G$
acts anti-diagonally, and a consideration of degrees shows that the
induced rational map from $(\CC\times\CC)/G$ to $\YY$ is birational.
\end{proof}

We note that $(\CC\times\CC)/G$, and
therefore $\XX$, contains infinitely many rational curves.  Indeed the
images in the quotient of the graphs of $q^n$-power Frobenius maps
$\CC\to\CC$ and their transposes are rational curves.  This gives a
Zariski dense set of rational curves on $\XX$.

Note that when $d=q^n+1$, the image of the graph of the $q^n$-power
Frobenius $\CC\to\CC$ in $\XX$ is the section of $\XX\to\P^1$
corresponding to the point $(t,h(t)^{d/r})$, and the image of the
transpose of Frobenius corresponds to the point $(t^{d-1},h(t)^{d/r})$.
In some sense, this ``explains'' these points.

Our next result shows that $\XX$ has general type as soon as $\CC$ has
genus $>1$.  (See also \cite[\S7.1]{Granville07} for another proof of
this fact.)  If $h$ has degree $s$ with $(r,s)=1$ and distinct,
non-zero roots, and if $r|d$, then the genus of $\CC$ is
$(r-1)(ds-2)/2$ which is $>1$ for large $d$ as soon as $r>1$ and
$s\ge1$.

\begin{lemma}\label{lemma:general-type}
  Let $C$ be a curve of genus $g(C)>1$ over a field $k$.  Let $G$ be a
  finite abelian group of automorphisms of $C$ with the order of $G$
  prime to the characteristic of $k$.  Let $Y = C \times C$ and let
  $G$ act on $Y$ ``anti-diagonally": $g(y_1,y_2) = (gy_1,g^{-1}y_2)$.
  Then the quotient $Y/G$ is of general type.
\end{lemma}

Note that $Y/G$ is normal with isolated singular points, so it makes
sense to speak of the canonical bundle and the plurigenera of $Y/G$. 

\begin{proof}
  We will show that $Y/G$ has Kodaira dimension 2, i.e., that the
  plurigenera of $Y/G$ grow quadratically.  Let
  $V_n=H^0(C,K_C^{\otimes n})$.  Since $g(C)>1, \dim V_n$ grows
  linearly with $n: \dim V_n \ge cn$ for some $c>0$.

  Decompose $V_n$ into eigenspaces for the action of $G$.  At least
  one of them has dimension $\ge \dim(V_n)/|G|$.  Call it $V_{n,\rho}$
  (where $\rho$ is the character by which $G$ acts on this subspace).

  Since $G$ acts anti-diagonally, the image of
  $V_{n,\rho}\otimes V_{n,\rho} \to H^0(Y,K_Y^{\otimes n})$ (via
  pull-back and wedge product) lands in the $G$-invariant subspace,
  which we denote $H^0(Y,K_Y^{\otimes n})^G$.  The map is injective,
  so 
$$\dim H^0(Y,K_Y^{\otimes n})^G\ge(\dim(V_n)/|G|)^2.$$
  This last expression is $\ge c'n^2$ for some $c'>0$.

  Since $|G|$ is prime to the characteristic of $k$, we have
$$H^0(Y/G,K_{Y/G}^{\otimes n})=H^0(Y,K_Y^{\otimes n})^G.$$  
Thus $\dim H^0(Y/G,K_{Y/G}^{\otimes n}) \ge c'n^2$, as required.
\end{proof}

Now we show that $X$ is not isotrivial, i.e., there does not exist a
curve $X_0$ defined over a finite field $k$ and an isomorphism 
$$X\times_{\Fq(t)}\overline{\Fq(t)}\cong X_0\times_k\overline{\Fq(t)}.$$

\begin{prop} 
  The curve $X=X_{h,r,a}$ is not isotrivial for any
  $a\in\Fq(t)\setminus\Fq$.
\end{prop}

\begin{proof}
From the definition of isotrivial, it clearly suffices to prove that $X_{h,r,t}$ is
not isotrivial, so we assume $a=t$ for the rest of the proof.  We will
use the domination of a regular proper model $\XX$ of $X$ by
$\CC\times\CC$ where $\CC$ is the curve associated to $w^r=h(z)$.  

Let $Z\subset\CC\times\CC$ be the locus where $z_1z_2=0$.  Since
$h(0)\neq0$, this is the union of $2r$ curves each isomorphic to $\CC$
meeting transversally at $r^2$ points.

Let $\widetilde{\CC\times\CC}$ be the blow up of $\CC\times\CC$ at the
closed points where either $(z_1=0,z_2=\infty)$ or $(z_1=\infty,z_2=0)$.  Let
$\tilde Z$ be the strict transform of $Z$ in
$\widetilde{\CC\times\CC}$.

The anti-diagonal action of $G:=\mu_r$ on $\CC\times\CC$ lifts
uniquely to $\widetilde{\CC\times\CC}$, it preserves $\tilde Z$, and
it has no fixed points on $\tilde Z$.  (Again we use that
$h(0)\neq0$.)  It follows that $\tilde Z/G$ is the union of two copies
of $\CC$ meeting transversally at $r$ points.  In particular,
$\tilde Z/G$ is a semistable curve.  It also follows that
$\widetilde{\CC\times\CC}/G$ is regular in a neighborhood of
$\tilde Z/G$.

Let $\CC\times\CC\ratto\P^1_t$ be the rational map defined by
$t=z_1z_2$.  This induces a morphism $\widetilde{\CC\times\CC}\to\P^1$
which factors through $\pi:\widetilde{\CC\times\CC}/G\to\P^1_t$.
Moreover, the generic fiber of $\pi$ is $X$, and $\pi^{-1}(0)$ is
precisely $\tilde Z/G$.

We have thus constructed a regular proper model of $X$ in a
neighborhood of $t=0$ such that the special fiber is a non-smooth,
semi-stable curve.  This proves that the moduli map
$\P^1_t\to\overline{\mathcal M}_g$ associated to $X$ is non-constant, and
so $X$ is non-isotrivial.
\end{proof}

\section{Height bounds}
The finiteness of $X(\Fq(t))$ when $X$ has genus $>1$ is of course a
consequence of the Mordell conjecture for function fields. We will use
the ABC theorem to give a direct, effective proof of this fact for a
subclass of the curves studied above, namely a common generalization
of the curves in \cite{ConceicaoUlmerVoloch12} and \cite{AIMgroup}.

For the rest of the paper, we fix positive integers $r$ and $s$ prime
to one another and to $p$, we let $h(x)=x^s+1$, and we study the curve
$$X:\quad y^r=h(x)h(t^d/x)=\frac{(x^s+1)(x^s+t^{ds})}{x^s}$$
over $\Fq(t)$ where $d$ is prime to $p$.  As noted above, the
genus $g$ of $X$ is $(r-1)s$.

Note that if $(x,y)$ is an $\Fq(t)$-rational point on $X$ and $x$ is a
$p$-th power, then $(t^d/x,y)$ is another point and $t^d/x$ is not a
$p$-th power.

In this section, we prove the following height bound.

\begin{thm}\label{thm:ht-bound}
  Suppose that the genus $g$ of $X$ is $>2$.  Let $(x,y)$ be an
  $\Fq(t)$-rational point on $X$, write $x=u/v$ with $u,v\in\Fq[t]$,
  $(u,v)=1$, and let $\delta=max\{\deg u,\deg v\}$.  If $x$ is not
  a $p$-th power, then
$$\delta\le\frac{dg-1}{g-2}$$
  and if $x$ is a $p$-th power, then
$$\delta\le\frac{2d(g-1)-1}{g-2}.$$
\end{thm}

\begin{proof}
  The case when $x$ is a $p$-th power follows immediately from the
  case when $x$ is not a $p$-th power after replacing $x$ with
  $t^d/x$, so we may assume $x$ is not a $p$-th power.

We write $a$ for $t^d$.  The hypotheses imply that
$$\frac{(u^s+v^s)(u^s+a^sv^s)}{u^sv^s}$$
is an $r$-th power in $\Fq(t)$.  Since $u$ and $v$ are relatively
prime, we have
$$\left.\gcd(u^s,u^s+a^sv^s)\right|a^s$$
and 
$$\left.\gcd(u^s+v^s,u^s+a^sv^s)\right|(a^s-1),$$
and all of the other terms in the displayed quantity are pairwise
relatively prime, i.e.,
$$\gcd(u^s,v^s)=\gcd(u^s,u^s+v^s)=
\gcd(v^s,u^s+v^s)=\gcd(v^s,u^s+a^sv^s)=1.$$
Therefore, $v$ is an $r$-th power, $t^iu$ is an $r$-th power for some
$i\in\{0,\dots,r-1\}$, and $f(u^s+v^s)$ is an $r$-th power for some
$f$ dividing $(a^s-1)^{r-1}$.

Next, we recall the ABC theorem in the following form (a special case
of \cite[Chapter 6, Lemma 10]{Mason84}).

\begin{ABC}
  If $A,B \in \F_q[t]$ are not both $p$-th powers, $(A,B)=1$, and
  $C=A+B$, then we have
$$\max \{\deg A,\deg B, \deg C \} \le \deg N(ABC) -1,$$ 
where $N(P)$ is the product of irreducible factors of $P$.
\end{ABC}

Apply this with $A=u^s$, $B=v^s$.  We have
$\deg N(A)\le(\delta+r-1)/r$, $\deg N(B)\le \delta/r$, and
$$\deg N(C)\le(\delta s+\deg f)/r\le(\delta s+ds(r-1))/r.$$
Since $A$ and $B$ are relatively prime, $N(ABC)=N(A)N(B)N(C)$ and we
find that
$$\delta s\le\frac{\delta(s+2)+ds(r-1)-1}{r}$$
and so
$$\frac{\delta((r-1)s-2)}{r}\le\frac{ds(r-1)-1}{r}.$$
Assuming that $g-2=(r-1)s-2>0$, we find that
$$\delta\le\frac{ds(r-1)-1}{(r-1)s-2}=\frac{dg-1}{g-2}$$
as desired.
\end{proof}

We note that when $d=p^n+1$, we have points on $X(\F_p(t))$ with $x$
coordinate equal to $t^{(p^n+1)/(p^m+1)}$, which are not $p$-th
powers, and for $m|n$, with $n/m$ odd, equal to
$t^{(p^n+1)p^m/(p^m+1)}$, which are $p$-th powers.  This shows that no
major improvement of the inequality of the theorem can be expected.

\section{Cardinality bounds}
We continue to study the curve 
$$X:\quad y^r=\frac{(x^s+1)(x^s+t^{ds})}{x^s}$$
over $\Fq(t)$ where $p$, $r$, and $s$ are pairwise relatively prime
and $d$ is prime to $p$.  Theorem~\ref{thm:ht-bound} yields an explicit
bound on the number of points on $X(\Fq(t))$ which is independent of
$q$:

\begin{cor}
  There is a constant $C$ depending only on $r$ and $s$, such that,
  for any power $q$ of $p$ and any $d$ prime to $p$, we have
  $\# X(\F_q(t)) \le C^d$.
\end{cor}

\begin{proof}
  Theorem~\ref{thm:ht-bound} shows that the $x$ coordinate of any
  affine point has degree $O(d)$ and likewise the $y$ coordinate.
  A curve that has infinitely many points of bounded
  height (with coefficients in the algebraic closure of $\F_q$) 
  is isotrivial by \cite[Proposition 2]{Lang60}
  and the remark that immediately follows,
  so we get finiteness this way without appealing to the
  Mordell conjecture. But we get more: The conditions on the $O(d)$
  coefficients of the numerator and denominator of $x$ and $y$ for the
  point to lie on $X$ is a system of $O(d)$ equations in $O(d)$
  variables and each equation has degree at most $r+2s$. We can
  consider this system over the algebraic closure of $\F_p$ and, by
  the above argument, it has finitely many solutions, so by B\'ezout's
  theorem it has at most $(r+2s)^{O(d)}$ solutions, proving the
  corollary.
\end{proof}

The main result of \cite{PachecoPazuki13} implies a bound similar to
that of the theorem but with $C$ depending on $r$, $s$, and $p$. The
cardinality of the set of points constructed in
\cite{ConceicaoUlmerVoloch12} (and reviewed in Section~1 above) when
$d=p^n+1$ is bounded by a multiple of the number of divisors of $n$,
so there is a huge gap between the known upper bounds for the number
of points and the number of points we can produce. It would be very
interesting to narrow this gap or perhaps identify all the rational
points.

Finally, we note that it is possible to improve the exponent when $d$
is large with respect to $q$.  Indeed, the degree of conductor of the
Jacobian of $X$ is $O(d)$ (with a constant depending only on $r$ and
$s$).  It follows from the arguments in \cite[\S11]{Ulmer07b}
(generalizing \cite{Brumer92}) that the order of vanishing at $s=1$ of
the $L$-function of $X$, and therefore the rank of the Mordell-Weil
group of the Jacobian of $X$, is $O(d/\log d)$ (with a constant
depending on $r$, $s$, and $q$).  Applying \cite{BuiumVoloch96}, we
find that the number of points on $X$ is at most $C_1^{d/\log d}$
where $C_1$ depends on $r$, $s$, and $q$.  These bounds, and in
particular the exact value of the rank, can in many cases be
determined more precisely using the domination by a product of curves
in Section~3 and arguments as in \cite{Ulmer13a}.

\emph{ Acknowledgements:} Both authors thank the Simons Foundation for
financial support under grants \#359573 and \#234591.  We also thank
Igor Shparlinski for comments on an earlier version of the paper.

\bibliography{database}{}
\bibliographystyle{alpha}

\end{document}